\documentclass[reqno]{amsart}
\usepackage[utf8]{inputenc}
\usepackage[T1]{fontenc}
\usepackage{amsmath,amsthm,enumitem,mathrsfs}
\usepackage{amssymb,xcolor,graphicx,multicol}
\usepackage[english]{babel}
\usepackage{caption}
\usepackage{subcaption}

\newcommand{\ve}{\varepsilon}
\newtheorem{thm}{Theorem}
\newtheorem{cor}{Corollary}
\newtheorem{lemma}{Lemma}

\newcommand{\ka}{\kappa}

\theoremstyle{remark}
\newtheorem{rmk}{Remark}

\usepackage{etoolbox}
\AtEndEnvironment{proof}{\setcounter{case}{0}}

\begin{document}

\title[On the Zoll deformations of the Kepler problem]{On the Zoll deformations of the Kepler problem}

\author[L. Asselle]{Luca Asselle}
\address{Ruhr-Universit\"at Bochum, Universit\"atsstra\ss e 150, 44801, Bochum, Germany}
\email{luca.asselle@math.uni-giessen.de}

\author[S. Baranzini]{Stefano Baranzini}
\address{Universit\'a degli Studi di Torino, Facolt\'a di Matematica, Via Carlo Alberto 10, 10123, Torino, Italy}
\email{stefano.baranzini@unito.it}

\date{\today}
\subjclass[2000]{....}

\begin{abstract}
A celebrated result of Bertrand states that the only central force potentials on the plane with the property that all bounded orbits are periodic are the Kepler potential and the potential of the harmonic oscillator. 
In this paper, we complement Bertrand's theorem showing the existence of an infinite dimensional space of central force potentials on the plane which are \textit{Zoll} at a given energy level, meaning that
 all non-collision orbits with given energy are closed and of the same length. We also determine all natural systems on the (not necessarily flat) plane which are invariant under rotations and Zoll at a given energy and prove several existence and rigidity results for systems which are Zoll at multiple energies. 
 \end{abstract}

\maketitle

%%%%%%%%%%%%%
%%%%%%%%%%%%%
%%%%%%%%%%%%%

	\section{Introduction}
	
The Kepler problem has been investigated now for more than three centuries and the contributions of many physicists and mathematicians shed light on many of its peculiar properties. For instance, negative energy solutions are ellipses, hence closed, and their period is a function of the semi-major axis. The system is \emph{super integrable}, meaning that there are more integral of motions than the dimension of the configuration space. Also, despite being singular, %, being described by the Lagrangian
	%\begin{equation}
	%	\label{eq:lagrangian_cartesian}
	%	L(v,x) = \frac12 \vert v \vert^2+\frac{1}{\vert x\vert},
	%\end{equation}
    the system can be regularized in at least two ways: for each energy level there is a branched two-cover conjugating the Kepler problem to the harmonic oscillator, and a diffeomorphism conjugating it with the geodesic flow on the round sphere.
    
    In this paper we focus primarily on the first property, which hereafter we refer to as the \emph{Zoll property}: non collision orbits are closed and foliate each negative energy level. Historically, Zoll systems (especially Zoll metrics) have been investigated starting with the pioneering work of Darboux \cite{Darboux:1877} and Zoll \cite{Zoll:1903} and have since then been of key importance: indeed, energy levels of Hamiltonian systems for which all orbits are periodic locally optimize symplectic capacities among level sets with fixed volume, see e.g. \cite{Abbo:2018,Abbo:2020, Abbo:2023,Paiva:2014} and reference therein. 
    A consequence of this optimal property is that the
action of periodic orbits of a periodic Hamiltonian flow depends only on the phase space and not on the particular
flow. In view of that, it is a significant problem
to find examples (in a given concrete class) of Hamiltonian systems with the property that all orbits with a given
energy are periodic. In this paper we address such a question in the case of radial potential flows: \emph{How special is the Zoll property of the Kepler problem? Can we perturb the system preserving the Zoll property at least on some negative energy level?}
    
    The first answer in this direction was given by Bertrand \cite{Bertrand:1873} who showed that the only analytic radial potentials on $\mathbb{R}^2$ (with the euclidean metric) such that all bounded orbits are periodic are $ 1/\vert x\vert$, the Kepler potential, and $-\vert x\vert^2$, the potential of the harmonic oscillator.
    A few years later, Darboux constructed two finite dimensional families of analytic perturbations of the Kepler problem having only periodic trajectories on all energy levels. Clearly, the metric is not the euclidean one except for the Kepler problem. For further details we refer to \cite{Darboux:1877,Albouy:2022,Santoprete2008}. 
    
    If one is interested on a single energy level only, say -1, then the space of deformations of the Kepler problem having the Zoll property at energy -1 (meaning that all non collision orbits with energy -1 are periodic) is obviously infinite dimensional. As observed by Larmor \cite{Larmor} and Routh \cite{Routh}, there is an infinite dimensional space of \textit{projective transformation} which modify the dynamics of the Kepler problem at energy -1 only by time reparametrization. However, the metric is not flat in any of the deformed systems. In particular, the question of the existence of central force potentials on the euclidean plane different from the Kepler or the harmonic one and having the Zoll property on the energy level -1 remains open. In this paper we will give a positive answer to this question (Corollary 2), and more generally construct, for any $n \in \mathbb{N}$, infinite dimensional families of deformations of the Kepler problem which satisfy the Zoll property on $n$ distinct negative energy levels. 
    
First, we determine the complete deformation space at one fixed energy level in Theorem \ref{thm:complete_deformation_space} and derive some consequences out of it. In Corollary 1 we show that the space of smooth rotationally invariant metrics $g$ on $\mathbb R^2$ such that the natural system determined by $g$ and the Kepler potential is Zoll at one given energy, say $-1$, is infinite dimensional. Such an abundance of geometries leading to a Zoll system at  energy $-1$ disappears if we require the metric to be analytic. In fact, in Theorem 2 we show that there is no analytic metric of revolution on the plane but the flat metric giving rise (together with the Kepler potential) to a natural system which is Zoll at energy $-1$. In Corollary 2 we then apply Theorem \ref{thm:complete_deformation_space}  to find ``exotic'' central force potential on the plane having the Zoll property at energy $-1$. 

In the second part of the paper, we build a deformation space for multiple energy levels which are ``not too close'' to each other, see Theorem \ref{thm:deformation_fixed_energies} and \ref{thm:existencearbitrary}, and finally discuss some rigidity and flexibility phenomena which occur when considering energy levels which are close to each other, see Corollaries 3-5.     

\vspace{2mm}

\noindent \textbf{Acknowledgments:} L.A. and S.B. warmly thank Gabriele Benedetti for many discussions about the content of the paper.	The second author is partially supported by INDAM-GNAMPA research group and by the PRIN 2022 project 2022FPZEES -- \emph{Stability in Hamiltonian Dynamics and Beyond}.

%%%%%%%%%%%%%
%%%%%%%%%%%%%
%%%%%%%%%%%%%
	
	\section{Main results}
	 \label{sec:main_results}
	
   \subsection{Preliminaries and background}
	Given a smooth manifold $M$ we say that a Hamiltonian $H$ on $T^*M$ is \emph{natural} if for any $q \in M$ the restriction of $H$ to $T^*_qM$ is a positive definite quadratic form. We additionally require $H|_{T^*_qM}$ to be bounded from below if $M$ is non compact. Natural Hamiltonians are of the form:
	\begin{equation}
		\label{eq:natural_ham}
		H(q,p) = \frac{1}{2}g_q(p,p)+V(q),
	\end{equation}
    where $g$ is a Riemannian metric on $M$ and $V$ a smooth function. 	
	To any such $H$ we can associate a Lagrangian function $L:TM\rightarrow \mathbb R$,  which we refer to as \emph{natural Lagrangian}, via the Legendre transform. 
	
	Hereafter, we identify a natural system with the corresponding pair $(g,V)$. 
		Natural systems are sometimes called mechanical since the Euler-Lagrange equations of $L$ describe the motion of a particle subject to the influence of the potential $V = L|_{M_0}$ (here $M_0$ stands for the zero section of $TM$).     
    Let $H$ be a natural Hamiltonian as in \eqref{eq:natural_ham}. It is well known that, for fixed $h\in \mathbb R$, the dynamics on the level set $\{H=h\}$ is equivalent to the geodesic dynamics of the \textit{Jacobi-Maupertuis metric}  
    \[
    ({g}_{JM})_q(p,p) = 2(h-V(q))g_q(p,p)
    \]
    up to time reparametrization. 
    Notice that $g_{JM}$ is degenerate if $h\le \sup_q V(q).$  See e.g. \cite{Montgomery} for more details. 
     Clearly, all solutions of Hamilton's equations contained in $\{H=h\}$ are periodic if and only if all geodesics of $g_{JM}$ are closed.%This observation will be exploited in Sections \ref{sec:proof_one_level}-\ref{sec:proof_rigidity} to prove the main statements. 
    
    Riemannian metrics with the property that all geodesics are closed and with the same length, usually known as \textit{Zoll metrics}, 
 have been extensively studied starting from the work of Darboux \cite{Darboux:1877} and Zoll \cite{Zoll:1903} in the last part of the nineteenth century. Among all closed orientable surfaces, 
 the sphere $S^2$ is the only one supporting Zoll metrics. Surprisingly though, $S^2$ admits an abundance of Zoll metrics: their moduli space is infinite dimensional 
and its tangent space at the standard metric has been computed first formally by Funk \cite{Funk:1913} and then rigorously by Guillemin \cite{Guillemin:1976}. Recently, an higher dimensional version of Guillemin's result has been proved by Ambrozio et. al. in \cite{Ambrozio:2021}. A magnetic version for systems on the two-torus is given in \cite{Asselle:2023}.
    
    In this paper, we will focus on the Jacobi-Maupertuis metrics corresponding to negative energy surfaces of the Kepler problem, which we recall is described by the natural Lagrangian
    	\begin{equation}
		\label{eq:lagrangian_cartesian}
		L(v,x) = \frac12 \vert v \vert^2+\frac{1}{\vert x\vert},
	\end{equation}	
	Let us observe that, for a given energy $-h/2$, the energy shell $\{H = -h/2\}$ is not compact. Topologically it is a product $D\times S^1 $ where $D$ is an open disk. Moreover the so called \emph{Hill's region}, i.e. the projection of the energy shell onto the base, is diffeomorphic to the set $\mathcal H_0 =\{x\in \mathbb{R}^2: 0<\vert x\vert\le 1\}$. The metrics induced by (deformations of) the Lagrangian in \eqref{eq:lagrangian_cartesian} via Jacobi-Maupertuis principle are non complete, singular at the origin and degenerate at the boundary of the Hill's region (the so called \emph{Hill's boundary}). However, they share an additional $S^1$ symmetry coming from the rotational invariance. Thanks to this and to the fact that the metrics are degenerate on the Hill's boundary, it is possible to adapt the classical classification result for rotationally invariant metrics on $S^2$ to this setting (see \cite{Besse:book}), thus showing that they are in one-to-one correspondence with odd functions on $[-1,1]$ with suitable boundary conditions. 
    This paper is devoted to the construction of infinite dimensional families of rotationally invariant natural systems close (in the $\mathscr{C}^\infty$ sense) to the Kepler problem which are \textit{Zoll}, meaning that all non-collision trajectories on some fixed energy levels are periodic. 
    In the remaining part of this section we present the main results of the paper. %In Section \ref{sec:one_level} we build a complete deformation space which are Zoll at a fixed energy level. In Section \ref{sec:more_levels} we build an infinite dimensional space of Zoll deformation for any finite number of energy levels. 

	%%%%%%%%%%%%%%%
	
	\subsection{Zoll deformations at one fixed energy level}
	\label{sec:one_level}
	
	In this section we focus on the deformations of the Kepler problem through rotationally invariant natural systems $(g,V)$ which are Zoll at one fixed (negative) energy. Thus, fix $h>0$ and consider the Euler-Lagrange flow induced by the Lagrangian of the Kepler problem at energy $-h/2$, which (in polar coordinates) can be seen as the flow of the natural Lagrangian
		\begin{equation}
		L_{h,\mathrm{Kep}}:T(\mathbb R^2 \setminus \{0\})\to \mathbb R,\quad L(v_\rho,v_\theta,\rho,\theta) = \frac12 (v_\rho^2 + \rho^2 v_\theta^2) + \frac 1\rho - \frac h2,
		\label{eq:Lagrangiankepler}
		\end{equation} 
		at energy $\{H=0\}$. 
		
		We set 
		$$\mathscr{C}^{\infty}_{\mathrm{odd},0}([-1,1]):=\Big \{ f \in \mathscr{C}^\infty([-1,1])\ \Big |\ f \ \text{odd}, \ f(1)=f(-1)=0\Big \}$$
		and, for $I,J\subset \mathbb R$ small open intervals containing 0, we consider smooth 1-parameter families of maps 
		\begin{align*} 
		&I\to \mathscr{C}^{\infty}_{\mathrm{odd},0}([-1,1]), \quad \epsilon \mapsto f_\epsilon, \quad \text{such that} \quad f_0 = 0, \\ 
		&J\to \mathscr{C}^\infty([0,2/h]), \quad \tau \mapsto \varphi_\tau, \quad \text{such that} \quad \varphi_0 = 0,
		\end{align*}
		and the corresponding family of natural Lagrangians 
	   $(L_{h,f_\ve,\varphi_\tau})$ given by 
		\begin{align}L_{h,f_\ve,\varphi_\tau } & (v_\rho,v_\theta,\rho,\theta) \nonumber \\ &= \frac{e^{- \varphi_\tau(\rho)}}{2}\left( \left(1+ \frac{f_\ve (1-h\rho)}{2-h\rho}\right)^2 v_\rho^2+\rho^2v_\theta^2\right) + e^{\varphi_\tau(\rho)}\left (\frac 1{\rho} - \frac h2\right). \label{eq:deformationkepler}
	\end{align}
	
	\begin{rmk}
 With slight abuse of notation, in \eqref{eq:deformationkepler} we denote by $f_\ve$ resp. $\varphi_\tau$ any smooth extension of $f_\ve$ to $(-\infty,1]$ resp. of $\varphi_\tau$ to $[0,+\infty)$. 
 Notice that changing $f_\ve$ outside $[-1,1]$ or $\varphi_\tau$ outside $[0,2/h]$ does not change the dynamics of $L_{h,f_\ve,\varphi_\tau}$ at energy $\{H=0\}$ 
 because the changes take place outside the Hill's region.
 \end{rmk}
 
	Notice that $L_{h,f_0,\varphi_0}=L_{h,\text{Kep}}$ and that the Jacobi-Maupertuis metric of $L_{h,f_\ve,\varphi_\tau}$ at energy $\{H=0\}$ is given by 
	\begin{equation}
		\label{eq:JM_metric_deformed}
		g_{JM}= \left (\frac{2-h\rho}{\rho}\right )\left (\left (1+ \frac{f_\ve(1-h\rho)}{2-h\rho}\right)^2 \mathrm{d}\rho^2 + \rho^2 \mathrm{d}\theta^2\right),
	\end{equation}
	in particular, it is independent of $\varphi_\tau$. 
	Finally we denote by $g_{f_\ve,\varphi_\tau}$ the metric on $\mathbb R^2\setminus \{0\}$ defining $L_{h,f_\ve,\varphi_\tau}$, namely 
	\begin{equation}
	 \label{eq:deformed_metric}	 
	 g_{f_\ve,\varphi_\tau} (\rho,\theta) := e^{-\varphi_\tau(\rho)} \left ( \left(1+ \frac{f_\ve(1-h\rho)}{2-h\rho}\right)^2 \mathrm d \rho^2+\rho^2\mathrm d \theta^2\right).
	\end{equation}
	
	\begin{thm}
		\label{thm:complete_deformation_space} Let $h>0$ be fixed. The families $(L_{h,f_\ve,\varphi_\tau})$ given in \eqref{eq:deformationkepler}
    are the only possible smooth deformations of \eqref{eq:Lagrangiankepler} through natural Lagrangians which are rotationally invariant and such that the induced Euler-Lagrange flow at energy $\{H=0\}$ is Zoll.	Moreover, $g_{f_\ve,\varphi_\tau}$ extends smoothly to $\mathbb R^2$ if and only if 
    $$f_\ve^{(2k+1)}(1)=\frac{2k+1}{2}f_\ve^{(2k)}(1), \ \ \varphi_\tau^{(2k+1)}(0) =0, \quad \forall k\in \mathbb N_0.$$
	\end{thm}
	
	\begin{rmk}
	Combining Theorem 1 with the so called Darboux inversion \cite{Darboux:1889} (see also \cite{Albouy:2022}) one obtains all possible deformations of the harmonic oscillator through rotationally invariant  natural systems which are Zoll at one fixed (positive) energy. The details are left to the reader. 
	\end{rmk}
       
    Theorem \ref{thm:complete_deformation_space} will be proved in Section \ref{sec:proof_one_level}. As already observed, deformations of the flat metric or of the Kepler potential  supported outside the Hill's region $\mathcal H_0$ are dynamically uninteresting. For this reason, we call two metrics resp. two central force potentials \textit{equivalent} if they agree on $\mathcal H_0$. The next two results, which will also be proved in Section \ref{sec:proof_one_level}, are rather immediate corollaries of Theorem 1. 
    
    \begin{cor}
	\label{cor:metric_deformations} For every $h>0$ fixed there is an infinite dimensional space of non-equivalent rotationally invariant smooth Riemannian metrics $g$ on the plane such that the Euler-Lagrange flow 
	of $(g,V_{\mathrm{Kep},h})$ is Zoll at energy $\{H=0\}$. 
    \end{cor}
    
    None of the metrics in Corollary 1 is analytic. Indeed, we have the following rigidity statement. 
    
         \begin{thm}
   	\label{thm:rigidity}
   	Let $h>0$ be arbitrary. If $g$ is a real analytic rotationally invariant Riemannian metric on the plane such that the system $(g,V_{\mathrm{Kep},h})$ is Zoll on
	 $\{H=-\frac h 2\}$, then $g$ is the flat metric. 
   \end{thm}  
   
      As originally shown by Darboux \cite{Darboux:1877}, there are two families (a 2- and a 3-parameter family) of mechanical systems on a surface of revolution having the property that all bounded 
    orbits are periodic. As it can be seen from \cite[Section 7]{Albouy:2022}, the corresponding metrics are always analytic. However, the corresponding central force potential is never the Kepler potential, unless $g$ is the flat metric.
    
    \begin{cor}
    	\label{cor:potential_deformations}
    For every $h>0$ fixed there is an infinite dimensional space of non-equivalent central force potentials $V$ on the plane  (containing $V_{\mathrm{Kep},h}$) such that the Euler-Lagrange flow 
	of $(g_{\mathrm{flat}},V)$ is Zoll at energy $\{H=0\}$. 
    \end{cor}

    %%%%%%%%%%%%%%
    
    \subsection{Zoll deformations at several energy levels}
    
    \label{sec:more_levels}

In this section we are interested in those natural systems $(g,V)$ which are Zoll at multiple energy levels. Our first result is about the existence of an infinite dimensional space of non-equivalent rotationally invariant metrics $g$ on the plane such that the system $(g,V_{\text{Kep}})$ is Zoll at two (negative) energies. 

    \begin{thm}
	\label{thm:deformation_fixed_energies}
	Let $h> \kappa >0$ be arbitrary. Then, the space of non-equivalent rotationally invariant metrics $g$ on the plane such that the system $(g,V_{\mathrm{Kep}})$ is Zoll on $\{H=-\frac \kappa 2\}$ and on
	 $\{H=-\frac h 2\}$ is infinite dimensional. 
    \end{thm}
    
We will prove the theorem above in Section \ref{sec:proof_many_levels} by showing that the space $\mathscr{C}^{\infty}_{\mathrm c}([0,1])$ of smooth functions having compact support in $(0,1)$ embeds in the desired space of metrics. Clearly, as to be expected from Theorem \ref{thm:rigidity} none of the metrics arising with this construction is analytic.

  The next result deals with the existence question of Riemannian metrics $g$ such that $(g,V_{\text{Kep}})$ is Zoll at an arbitrary (but finite) number of energies. 
  
  \begin{thm}
  \label{thm:existencearbitrary}
  Let $n\geq 3$ and let $h_1>h_2 > ... > h_n >0$ be such that $h_i\geq 2 h_{i+1}$ for all $i=2,...,n-1$. Then, the space of non-equivalent rotationally invariant metrics $g$ on the plane such that the system $(g,V_{\mathrm{Kep}})$ is Zoll on $\{H=-\frac{h_i} 2\}$, for all $i=1,...,n$, is infinite dimensional. 
  \end{thm}

Theorem \ref{thm:existencearbitrary} is a rather immediate corollary of Theorem \ref{thm:deformation_fixed_energies} and will be proved in Section \ref{sec:proof_many_levels}. We shall stress the fact that the statement of Theorem \ref{thm:existencearbitrary} holds only for finitely (but arbitrarily) many energy levels and cannot be extended to the case of an infinite number of energies. More details will be provided in Section \ref{sec:proof_many_levels}. If one drops the condition $h_i\geq 2 h_{i+1}$ in the theorem above then the situation becomes more involved, and indeed one has both rigidity and flexibility phenomena. We will illustrate this in Section \ref{sec:multipleclose}, where we will also prove the following

\begin{thm}
	Let $n\geq 3$, $h_1>h_2 > ... > h_n >0$ and $\ve >0$ be fixed. Then, for any $i=1,...,n$ there exists $h'_i$ $\ve-$close to $h_i$ such that the space of non-equivalent rotationally invariant metrics $g$ on the plane for which the system $(g,V_{\mathrm{Kep}})$ is Zoll on $\{H=-h_i'/ 2\}$, for all $i=1,...,n$, is infinite dimensional. 
\end{thm}  
    %%%%%%%%%%%%%
    %%%%%%%%%%%%%
    %%%%%%%%%%%%%
    
	\section{Proof of Theorem \ref{thm:complete_deformation_space}}
	\label{sec:proof_one_level}
	In this section we prove Theorem \ref{thm:complete_deformation_space} and Corollaries \ref{cor:metric_deformations} and \ref{cor:potential_deformations}.
	As we are dealing with a fixed energy level $\{H = -\frac{h}{2}\}$ for some $h>0$, for later purposes it is convenient to consider the \emph{shifted} Lagrangian 
	$$L(v,x) = \frac12 \vert v \vert^2+\frac{1}{\vert x\vert} - \frac h2$$ 
	and the induced Euler-Lagrange flow on the energy hypersurface $\{H=0\}$. 
	
	\begin{proof}[Proof of Theorem \ref{thm:complete_deformation_space}]
	We rewrite the shifted Lagrangian in polar coordinates $x = \rho e^{i\theta}$ and obtain
    \begin{equation}
    	\label{eq:lagrangian_polar}
    	L(v_\rho,v_\theta,\rho,\theta) = \frac12( v_\rho^2+\rho^2v_\theta^2) +\frac{1}{\rho} - \frac h2.
    \end{equation}

    Maupertuis' principle asserts that solutions of the Euler Lagrange equations given by $L$ on the energy hypersurface $\{H=0\}$ correspond up to time reparametrization to geodesics for the Jacobi-Maupertuis metric:
    \begin{equation}
    	\label{eq:JM_metric_Kepler_h}
    	g_{JM} = 2\left(-\frac{h}{2}+\frac1\rho\right)(d\rho^2+\rho^2d\theta^2) = \left(\frac{2-h \rho}{\rho}\right)(d\rho^2+\rho^2d\theta^2).
     \end{equation}
     
     We now rewrite \eqref{eq:JM_metric_Kepler_h} in the normal form given in \cite[Proposition 4.10 and Theorem 4.13]{Besse:book}.
     Let us observe that the coordinate $\rho$ lives in the interval $[0,\frac2h]$. Moreover, the function $\rho(2-h\rho)$ vanishes at the boundary of $[0,\frac2h]$ and has a unique maximum point $\rho =\frac{1}{h}$ where its value is $\frac1h$.
     Let us change coordinates to $(r,\theta)\in [0,\pi]\times [0,2\pi]$ where $r$ is determined by the relation
     \[
    (\sin r)^2 = h\rho(2-h\rho).\]
    Since $2h\rho-(h\rho)^2 = 1-(h\rho-1)^2$ and $\cos r$ is positive on $[0,\pi/2]$, we obtain
     \begin{equation*}
     	\cos r = 1-h\rho \quad \text{ and } \quad \sin r d r = h d \rho.
     \end{equation*}
     
    Summing up and rewriting \eqref{eq:JM_metric_Kepler_h} in these new coordinates, we have
    \begin{align*}
    	g_{JM} & = \frac{(2-h\rho)(\sin r)^2}{h^2\rho} dr^2 + (\sin r)^2 d \theta^2 \\ & =\frac1h\left( (2-h\rho)^2dr^2 + (\sin r)^2 d \theta^2\right) \\ & =\frac1h((1+\cos r)^2dr^2 + (\sin r)^2 d \theta^2) .
    \end{align*}
    
    Thanks to  \cite[Proposition 4.10 and Theorem 4.13]{Besse:book}, we know that smooth Zoll metrics of revolution on the (open) cylinder $(0,\pi)\times [0,2\pi]$ are given by:
    \begin{equation}
    	g = (1+f(\cos r))^2dr^2 + (\sin r)^2 d \theta^2,
	\label{eq:Tannery}
    \end{equation}
    where $f:(-1,1)\to (-1,1)$ is any smooth \emph{odd} function (for the Kepler problem we have $f=\text{id}|_{(-1,1)}$).  We now consider a smooth one-parameter family $\{f_\ve:(-1,1)\rightarrow \mathbb R\}$, 
    with $\ve$ sufficiently close to zero, such that $f_0=\text{id}|_{(-1,1)}$ and set 
    \begin{equation*}
    	g_{f_\ve} = \frac1h((1+\cos r+ f_\ve(\cos r))^2dr^2 + (\sin r)^2 d \theta^2).
    \end{equation*}
    We readily see that a necessary condition for $g_{f_\ve}$ to be well-defined for $|\ve|$ sufficiently small is that $f_\ve$ extend continuously to $f_\ve:[-1,1]\to \mathbb R$ and $f_\ve(1)=f_\ve(-1)=0$. 
    Switching back to the old coordinates $(\rho,\theta)$ and recalling that $\cos r = 1-h\rho $ we obtain
    \begin{align}    	
    	\label{eq:JM_metric_deformation_h}
    	g_{f_\ve}&=\frac{(2-h\rho+f_\ve(1-h\rho))^2}{\rho (2-h\rho)} d \rho^2+\rho(2-h\rho)d\theta^2 \nonumber \\
    	&=\left(\frac{2-h\rho}{\rho} \right) \left( \left(1+ \frac{f_\ve (1-h\rho)}{2-h\rho}\right)^2 d \rho^2+\rho^2 d\theta^2\right).
    \end{align}
    Using \eqref{eq:JM_metric_deformation_h} we see that $g_{f_\ve}$ (for $\ve$ small enough) extends to a smooth metric on $\{\rho>0\}$ (that is, on $\mathbb R^2\setminus \{(0,0)\}$) if and only if
    
    $$f_\ve\in \mathscr{C}^\infty_{\text{odd},0}([-1,1]) := \Big \{ f\in \mathscr{C}^\infty([-1,1], \mathbb R) \ \Big |\ f \ \text{odd}, \  f(-1)=f(1)=0\Big \}.$$
    
    For any such $f_\ve\in \mathscr{C}^\infty_{\text{odd},0}([-1,1])$, we recognize $g_{ f_\ve}$ in \eqref{eq:JM_metric_deformation_h} 
     as the Jacobi-Matupertuis metric of the system given by the Lagrangian function:
    \begin{equation}
    	\label{eq:deformed_lagrangian_metric}
    	L_{f_\ve}(v_\rho,v_\theta,\rho,\theta) = \frac12\left( \left(1+\frac{f_\ve(1-h\rho)}{2-h\rho}\right)^2 v_\rho^2+\rho^2v_\theta^2\right) +\frac{1}{\rho}- \frac h2,
    \end{equation}
    at energy $\{H=0\}$. 
    Notice that the deformations given by the family $L_{f_\ve}$ above change the phase portrait on the energy level $\{H=0\}$. There is a further class of deformations, usually referred to as \emph{projective}, which change the flow at energy $0$ only by a time reparametrization:   
  for any smooth radial function $\varphi:[0,+\infty)\to \mathbb R$ the Jacobi-Maupertuis metric associated with the Lagrangian
    \begin{equation*}
    	\frac{e^{-\varphi(\rho)}}{2}\left( \left(1+ \frac{f_\ve(1-h\rho)}{2-h\rho}\right)^2 v_\rho^2+\rho^2v_\theta^2\right) + e^{\varphi(\rho)} \left (\frac 1{\rho}- \frac h2\right )
    \end{equation*}
    at energy $0$ is precisely $g_{f_\ve}$. We have thus proved that the space of deformations of \eqref{eq:lagrangian_polar} preserving the Zoll property at energy $0$ is given by
     \begin{align}
     	\label{eq:deformed_lagrangian}
    	L_{f_\ve,\varphi} (v_\rho,&v_\theta,\rho,\theta) \nonumber \\ & = \frac{e^{-\varphi(\rho)}}{2}\left( \left(1+\frac{f_\ve(1-h\rho)}{2-h\rho}\right)^2 v_\rho^2+\rho^2v_\theta^2\right) + e^{\varphi(\rho)}\left (\frac 1{\rho} - \frac h2\right).
    \end{align}  
    where $\varphi\in \mathscr{C}^\infty ([0,+\infty))$, $f_\ve\in \mathscr{C}^\infty_{\text{odd},0}([-1,1])$, and $\ve$ is small enough. 

    Now we have to derive necessary and sufficient conditions for the metric $g_{f_\ve,\varphi}$ in \eqref{eq:deformed_metric} to extend smoothly at $0$. First we look at the conditions to be imposed on $\varphi$. Let us observe that the vector field $\partial_\theta = -y\partial_x+x\partial_y$ is smooth on $\mathbb{R}^2$. Thus, if $g_{f_\ve,\varphi}$ is smooth so does the norm of $\partial_\theta$. A straightforward computation yields:
    \begin{equation*}
    	g_{f_\ve,\varphi}(\partial_\theta,\partial_\theta) = \frac12\rho^2 e^{-\varphi(\rho)}.
    \end{equation*}    
    This implies that $\varphi$ should extend to a smooth function on $\mathbb{R}^2$ and so all odd derivatives of $\varphi$ should vanish at the origin. Let us consider now the vector field  
    $$X = e^{\varphi(\rho)/2}\rho \partial_{\rho} = e^{\varphi(\rho)/2}( x\partial_x+y\partial_y)$$ 
    which is again smooth. It follows that its $g_{f_\ve,\varphi}$-norm is:
   \begin{equation*}
	g_{f_\ve,\varphi}(X,X) = \frac12\rho^2 \left(1+ \frac{f_\ve(1-h\rho)}{2-h\rho}\right)^2.
   \end{equation*}    
   Thus if $g_{f_\ve,\varphi}$ is smooth all odd derivatives of the function 
    $$\rho \mapsto \left(1+ \frac{f_\ve(1-h\rho)}{2-h\rho}\right)^2$$ 
    should vanish at the origin. By induction we prove that this happens if and only if all odd derivatives of $f_\ve(1-h\rho)/(2-h\rho)$ vanish at the origin. Indeed, the following well-known formula for the derivatives of the composition of functions gives:
   \begin{equation*}
   	\frac{d^n}{d\rho^n} F(G(\rho)) = \sum_{\pi} F^{(\vert\pi\vert)}(G(\rho))\prod_{B\in \pi}G^{(\vert B\vert)}(\rho),
   \end{equation*}
    where $\pi$ is a partition of $\{1,\dots,n\}$, $\vert \pi\vert$ the total number of its blocks , $B$ a block of $\pi$ and $\vert B\vert$ its cardinality. Since in our case 
    
    $$F(\rho) = (1+\rho)^2, \quad G(\rho) =  \frac{f_\ve(1-h\rho)}{2-h\rho},$$ 
    we consider only $\pi$ of $\vert \pi\vert\le 2$. By induction, we assume the claim to hold up to order $2k-1$. For $s:=2k+1$ we obtain using the fact that $G(0)=0$:
    \begin{align*}
    	\frac{d^{s}}{d\rho^{s}} & \Big |_{\rho =0} F\circ G(\rho) \\
	&= 2 (1+ G(\rho)) \frac{d^{s}}{d\rho^{s}} G(\rho) + 2 \sum_{\pi : |\pi|=2} \prod_{B\in \pi}\frac{d^{\vert B\vert}}{d\rho^{\vert B\vert}} G(\rho) \Big |_{\rho =0}\\
	&= 2  \frac{d^{s}}{d\rho^{s}}\Big  |_{\rho =0} G(\rho) + 2 \sum_{\pi : |\pi|=2} \prod_{B\in \pi}\frac{d^{\vert B\vert}}{d\rho^{\vert B\vert}} \Big |_{\rho =0} G(\rho).
	%&= 2 	\frac{d^{s}}{dx^{s}}\left( \frac{f(1-h\rho)}{2-h\rho}\right) + 2\sum_{\pi:\vert \pi\vert=2} \prod_{B\in \pi}	\frac{d^{\vert B\vert}}{dx^{\vert B\vert}} \left( \frac{f(1-h\rho)}{2-h\rho} \right)\\
    %	&= 2 	\frac{d^{s}}{dx^{s}} \frac{f(1-h\rho)}{2-h\rho} =0.
    \end{align*}
    Now, it is readily seen that each product 
    $$\prod_{B\in \pi}\frac{d^{\vert B\vert}}{d\rho^{\vert B\vert}} \Big |_{\rho =0} G(\rho)$$
    vanishes by inductive hypothesis, as it contains exactly one lower order odd derivative. It follows that 
    $$	\frac{d^{s}}{d\rho^{s}} \Big |_{\rho =0} F\circ G(\rho) =  2  \frac{d^{s}}{d\rho^{s}}\Big  |_{\rho =0} G(\rho),$$
    thus proving the claim. 
     
     Therefore, a necessary condition for $g_{f_\ve,\varphi}$ to extend smoothly to $\mathbb R^2$ is that $f_\ve(1-h\rho)/(2-h\rho)$ extends to a smooth function on $\mathbb{R}^2$. 
     Finally, let us rewrite this last condition in terms of derivatives of $f_\ve$. Assuming that the $(2k-1)$-th derivative of $f_\ve(1-h\rho)/(2-h\rho)$ vanish at the origin, and using the fact that 
     $$\frac{d^n}{d\rho^n} \Big |_{\rho=0} \left (\frac{1}{2-h\rho}\right ) = \frac{h^n n!}{2^{n+1}}, \ \forall n\in \mathbb N,$$
     we compute 
     \begin{align*}
     		\frac{d^{2k+1}}{d\rho^{2k+1}} \Big |_{\rho=0} \left(\frac{ f_\ve(1-h\rho)}{2-h\rho}\right) &= \frac{h^2(2k+1)!}{4(2k-1)!} 	\frac{d^{2k-1}}{d\rho^{2k-1}} \Big |_{\rho =0}\left( \frac{f_\ve (1-h\rho)}{2-h\rho}\right)  \\ & +\frac{h^{2k+1}(2k+1)}{4}f_\ve^{(2k)}(1)-\frac{h^{2k+1}}{2}f_\ve^{(2k+1)}(1)\\ 
		&= \frac{h^{2k+1}(2k+1)}{4}f_\ve^{(2k)}(1)-\frac{h^{2k+1}}{2}f_\ve^{(2k+1)}(1) 
     \end{align*}
     by inductive hypothesis. This implies that $f_\ve$ satisfies
     $$f_\ve^{(2k+1)}(1) = \frac{2k+1}{2}f_\ve^{(2k)}(1), \ \ \forall k \ge0.$$ 
     The converse statement is straightforward and completely analogous.
   \end{proof}

  %  \subsection{Proof of Corollary \ref{cor:metric_deformations} and \ref{cor:potential_deformations}}
    
    \begin{proof}[Proof of Corollary \ref{cor:metric_deformations}]
    	Take $\varphi_\tau=0$ for all $\tau\in J$. Then $g_{f_\ve,0}$ as above reads 
    	$$g_{f_\ve,0} (\rho,\theta) =   \left(1+ \frac{f_\ve(1-h\rho)}{2-h\rho}\right)^2 \mathrm d \rho^2+\rho^2\mathrm d \theta^2$$
    	and the corresponding natural Lagrangian
    	$$\frac{1}{2}\left( \left(1+ \frac{f_\ve (1-h\rho)}{2-h\rho}\right)^2 v_\rho^2+\rho^2v_\theta^2\right) + \left (\frac 1{\rho} - \frac h2\right)$$ 
    	is precisely the one defined by $(g_{f_\ve,0},V_{\text{Kep},h})$. 
    \end{proof}
    
    \begin{proof}[Proof of Corollary \ref{cor:potential_deformations}]
    	We start from the natural system $(g_{f_\ve,0},V_{\text{Kep},h})$ 
    	and perform a change of coordinates which makes the metric conformal to the flat metric.
    	Let us consider a new coordinate $\sigma$ and set $ \rho = A_\ve(\sigma)$ for some function $A_\epsilon$. Using $\mathrm d\rho = A_\ve'(\sigma) \mathrm d\sigma$ we compute for the pull-back metric:
    	\begin{align*}
    		g_\ve &= \frac{1}{2}\left(\left(1+\frac{f_\ve(1-hA_\ve(\sigma))}{2-hA_\ve(\sigma)}\right)^2A_\ve'(\sigma)^2 \mathrm d\sigma^2+A_\ve^2(\sigma)\mathrm d\theta^2\right)\\ &=\frac{A_\ve^2(\sigma)}{2\sigma^2}\left(\frac{\sigma^2 A'_\ve(\sigma)^2}{A_\ve^2(\sigma)}\left(1+\frac{f_\ve(1-hA_\ve(\sigma))}{2-hA_\ve(\sigma)}\right)^2 \mathrm d\sigma^2+\sigma^2\mathrm d\theta^2\right) .
    	\end{align*}
    	From this we readily see that $g_\ve$ is conformal to the euclidean metric if and only if $A_\ve$ solves the equation
    	\begin{equation}
    		\label{eq:change_coordinates_ODE}
    		A_\ve'(\sigma) = \frac{A_\ve(\sigma)(2-hA_\ve(\sigma ))}{\sigma (2-hA_\ve(\sigma)+f_\ve(1-hA_\ve(\sigma)))}.
    	\end{equation}
    
    	 As we are not interested in determining all central force potentials $V$ such that $(g_{\text{flat}},V)$ is Zoll at energy $\{H=0\}$, but rather want only to show that the space of such potentials is infinite dimensional, we  assume hereafter that $f_\ve = 0 $ in a small neighbourhood of the origin for all $\ve$. In this case, picking $A_\ve$ such that $A'_\ve(0) =1$ we obtain a smooth change of coordinates.
    	 
    	  In the new coordinates $(\sigma,\theta)$  the Lagrangian reads: 
    	\begin{align}
    		L_{h,f_\ve,0}(v_\sigma, v_\theta,\sigma, \theta) = \frac{A_\ve^2(\sigma)}{2\sigma^2} \left( v_\sigma^2+\sigma^2v_\theta^2\right) + \left ( \frac 1{ A_\ve(\sigma)} - \frac h2\right ),
    	\end{align}
    	and hence choosing $\tau=\ve$ and $\varphi_\ve (\rho):= 2\log (\frac{\rho}{A^{-1}_\ve(\rho)})$ we obtain 
    	$$L_{h,f_\ve,\varphi_\ve}(v_\sigma, v_\theta,\sigma, \theta) =  \left( v_\sigma^2+\sigma^2v_\theta^2\right) + \frac{A_\ve^2(\sigma)}{\sigma^2}  \left ( \frac 1{ A_\ve(\sigma)} - \frac h2\right ).$$
    	The claim readily follows.
    \end{proof}
    
   \begin{proof}[Proof of Theorem \ref{thm:rigidity}]
    Assume that $f$ is an analytic odd function on $\mathbb R$ vanishing at $x=\pm 1$ and such that the metric $g_f$ given in \eqref{eq:JM_metric_deformation_h} is smooth on the whole $\mathbb R^2$. By Theorem 1 we have that
  \begin{equation}
  f(1)=0, \quad f'(1) = 0, \quad f^{(2k+1)}(1) = \frac{2k+1}{2} f^{(2k)}(1), \ \forall k \in \mathbb N.
  \label{eq:fanalytic}
  \end{equation}
   Assuming further that the radius of convergence of the series expansion of $f$ at $x=1$ is sufficiently large (in particular, so that $x=0$ is contained in the convergence radius, so that exchanging series is allowed) we can compute 
   \begin{align*}
   f(x) &= \sum_{k=1}^{+\infty} \Big ( \frac{f^{(2k)}(1)}{(2k)!} (x-1)^{2k} + \frac{f^{(2k+1)}(1)}{(2k+1)!} (x-1)^{2k+1}\Big )\\ 
    &=  \sum_{k=1}^{+\infty} \frac{f^{(2k)}(1)}{2 (2k)!} (x-1)^{2k} (x+1)\\
    &=  \sum_{k=1}^{+\infty} \frac{f^{(2k)}(1)}{2 (2k)!} (x+1)\sum_{j=0}^{2k} \binom{2k}{j} (-1)^{2k-j} x^j\\
    &=  \sum_{j=0}^{+\infty} (-1)^j x^{j} \sum_{k=\lceil \frac j2\rceil}^{+\infty} \binom{2k}{j} \frac{f^{(2k)}(1)}{2 (2k)!}  + \sum_{j=0}^{+\infty} (-1)^j x^{j+1} \sum_{k=\lceil \frac j2\rceil}^{+\infty} \binom{2k}{j} \frac{f^{(2k)}(1)}{2 (2k)!} \\
    &= \underbrace{\sum_{k=1}^{+\infty} \frac{f^{(2k)}(1)}{2 (2k)!}}_{:=(1)} +\sum_{j=1}^{+\infty} (-1)^j x^{j} \underbrace{\Big ( \sum_{k=\lceil \frac j2\rceil}^{+\infty} \binom{2k}{j} \frac{f^{(2k)}(1)}{2 (2k)!} - \sum_{k=\lceil \frac{j-1}2\rceil}^{+\infty} \binom{2k}{j-1} \frac{f^{(2k)}(1)}{2 (2k)!} \Big )}_{=:(2)}. 
   \end{align*}
   Since $f$ is odd, (1) has to vanish and (2) has to vanish too for all even $j\in \mathbb N$. For simplicity we replace $j$ with $2h$, $h\in \mathbb N$, and rewrite (2) as 
   \begin{align*}
   0 & =  \sum_{k=h}^{+\infty} \binom{2k}{2h} \frac{f^{(2k)}(1)}{2 (2k)!} - \sum_{k=h-1}^{+\infty} \binom{2k}{2h-1} \frac{f^{(2k)}(1)}{2 (2k)!}\\ 
   	&= \sum_{k=h}^{+\infty} \frac{f^{(2k)}(1)}{2 (2k)!} \Big (  \binom{2k}{2h} - \binom{2k}{2h-1}\Big ), \quad \forall h \in \mathbb N.
	\end{align*}
	
	We set $a:=(a_{k})_{k\in 2 \mathbb N}$ by $a_{k}:= \frac{f^{(k)}(1)}{2 (k)!}$ for all $k\in 2 \mathbb N$ and observe that $a\in s(2 \mathbb N)$, the space of rapidly decreasing sequences (in particular $a\in \ell^2(2 \mathbb N)$), 
        since the radius of convergence of the series expansion of $f$ at $x=1$ is by assumption larger than 1. We now define the (unbounded) linear operator 
	$$T:s(2 \mathbb N) \subset \ell^2(2 \mathbb N)\to \ell^2(\mathbb N), \quad u \mapsto Tu,$$
	where 
	\begin{align*}
	(Tu)_h &:= \sum_{k=h,\  k\  \text{even}}^{+\infty} a_k \Big (  \binom{k}{2h} - \binom{k}{2h-1}\Big ), \quad \forall h \in \mathbb N.
	\end{align*}
	It is elementary to check that, if the radius of convergence of the series expansion of $f$ at $x=1$ is larger than $e$, then $T$ is well-defined and maps $s(2\mathbb N)$ into $\ell^2(\mathbb N)$. Also, it is straightforward to see that $T$ is injective. Indeed, with respect to the standard Hilbert bases of $\ell^2(2 \mathbb N)$ and $\ell^2(\mathbb N)$ respectively it is represented by an (infinite) upper-triangular matrix having all diagonal entries different from zero, more precisely equal to $1-k$ for all $k\in 2\mathbb N$. In particular, the equation $Tu=0$ has only the trivial solution. This implies that all even derivatives of $f$ in $x=1$ vanish, and hence using \eqref{eq:fanalytic}, that all derivatives of $f$ vanish in $x=1$. 
	
The general case can be dealt with in a similar fashion. 
    \end{proof}

      \section{Proof of Theorem \ref{thm:deformation_fixed_energies}}
      \label{sec:proof_many_levels}
   
   \begin{proof}[Proof of Theorem \ref{thm:deformation_fixed_energies}]
   	Fix $h>\ka$ and consider the energy levels  $\{H=-\frac{h}{2}\}$ and  $\{H=-\frac{\ka}{2}\}$. The corresponding Hill's regions are two disks of radii $\frac{2}{h}$ and $\frac{2}{\ka}.$ 
   	
   	Assume that a deformation as given in \eqref{eq:deformed_lagrangian_metric} on the energy level $\{H=-\frac{h}{2}\}$ is fixed (for notational convenience we hereafter drop the subscript $\ve$) and denote by $\tilde{f}$ any extension of $f$ to the interval $[1-\frac{2h}{\ka},1]$. Using \eqref{eq:JM_metric_Kepler_h} and \eqref{eq:JM_metric_deformation_h} we see that $\tilde f$ induces the following Jacobi-Maupertuis metric on the disk of radius $\frac{2}{\ka}$:
    \begin{align*}
 	g_{JM} &= \frac{2-\ka\rho }{\rho}\left(\left(1+\frac{\tilde f \left(1-h \rho \right)}{2-h\rho }\right)^2d\rho^2+ \rho^2 d\theta^2\right)\\
 	&= \frac{2-\ka\rho }{\rho}\left(\left(1+\frac{\tilde f \left(1-h\rho \right)(2-\ka\rho)}{(2-h\rho )(2-\kappa\rho)}\right)^2d\rho^2+ \rho^2 d\theta^2\right).
    \end{align*}
    Clearly, $g_{JM}$ has only closed solutions on $\{H = -\frac{\ka}{2}\}$ (besides the collision orbits) if and only if the function:
    \begin{equation*}
    	F(1-\ka \rho) = \frac{ \tilde f \left(1-h\rho \right)(2-\ka\rho)}{2-h\rho}, \quad \rho \in \left[0,\frac{2}{\ka}\right],
    \end{equation*}
     is odd on $[-1,1]$. 
     We define $$\xi := \frac{h}{\ka}-1>0$$ 
     and consider the change of variable $x = 1-h\rho$. A straightforward computation shows that $1-\kappa \rho  = \frac{x +\xi}{\xi+1}$  and allows us to rewrite the above equation as follows:
     \begin{equation*}
     	F\left(\frac{x +\xi}{\xi+1}\right) = \frac{\tilde{f}(x)(x +2\xi+1)}{(\xi+1)(1+x)}, \quad  x\in[-1-2\xi,1]
     \end{equation*}
     Denote by $R_{-\xi}$ the reflection about $-\xi$. Imposing the odd parity of $F$ we find:
     \begin{align*}
     	F\left(-\frac{x +\xi}{\xi+1}\right)&=- \frac{\tilde{f}(x)(x +2\xi+1)}{(\xi+1)(1+x)}=  \frac{\tilde{f}(R_{-\xi}(x))(1-x)}{(\xi+1)(1+R_{-\xi}(x))}&=F\left(\frac{R_{-\xi}(x) +\xi}{\xi+1}\right).
     \end{align*}
    Observing that $x+2\xi+1 = 1-R_{-\xi}(x),$ we deduce that the odd parity of $F$ can be rewritten as a condition on the extension $\tilde{f}$. Indeed, $F$ is odd if and only if the function $G(x) = \tilde{f}(x)/(1-x^2)$ is odd with respect to the reflection $R_{-\xi}$. Therefore, we obtain for $\tilde f$ the following reflection laws:
    \begin{align}
    	\label{eq:parity_tilde_f}
    	 \frac{\tilde{f}(R_{-\xi}(x))}{1-R_{-\xi}(x)^2} & =- \frac{\tilde{f}(x)}{1-x^2}, \quad \forall x\in[-1-2\xi,1], \\
    	\tilde{f}(-x) &= -\tilde{f}(x), \quad \forall x \in [-1,1]. \label{eq:ftildeodd}
    \end{align}
    
    Furthermore, we readily see that if $G(x)$ is smooth, then $F$ and $\tilde f$ both vanish at $x=\pm1$.% Indeed, $f(x) = (1-x^2)G(x)$ and setting $y = 1-\ka \rho$ we readily obtain $F(y) = (\xi+1)(1-y^2) G((1+\xi)y-\xi)$.  
    We have now to consider three cases:    
    \vspace{2mm}
    
    {\textbf{Case 1($h>2\ka$).}} Let us observe that in this case $\xi>1$ and so the interval $[-1,1]\subset[-\xi, 1]$ and in this case the inclusion is strict. Thus, it is enough to choose a suitable extension of $f$ to $[-\xi,1]$ and then reflect it on the whole $[-1-2\xi,1]$. This imposes certain conditions on the derivatives of  $\tilde {f}$ in $-\xi$. Namely $\tilde{f}^{(2n)}(-\xi) =0$ for all $n \in \mathbb{N}$. Clearly, this is the case if $\tilde{f}$ is compactly supported in $(-\xi, 1)$.
    
    \vspace{2mm}
    
    {\textbf{Case 2 ($h=2\ka$).}} This case is completely analogous to the previous one except for the fact that now $\xi =1$ and so $[-\xi, 1] = [-1,1]$. Thus, it is natural to ask that the derivatives of $f$ vanish at all order in $-1$.
    
    \vspace{2mm}
    
    {\textbf{Case 3 ($\ka<h<2\ka$).}} 
     In this case $\xi \in (0,1)$ and the condition that $F$ be odd imposes further restrictions on $\tilde{f}.$ Let us observe that in this case, the function $G(x)$ is $2\xi-$periodic. Indeed, using \eqref{eq:parity_tilde_f} and \eqref{eq:ftildeodd} we compute
     \begin{align*}
     	G(x) =- G(R_{-\xi}(x))= - \frac{\tilde f( -x-2\xi)}{1+ (-x-2\xi)^2} = \frac{\tilde f(x+2\xi)}{1+ (x+2\xi)^2} = G(x+2\xi).
     \end{align*}
Since $G$ is also odd with respect to the reflections about $0$ and $-\xi$, we infer that $G$ is determined by its restriction to $[-\xi,0].$ In particular, any function compactly supported in $[-\xi,0]$ uniquely extends to a function $G$ on $[-1-2\xi,1]$ with the desired properties. Therefore, the space
 $\mathscr{C}^\infty_c([-\xi,0])$ injects in the space of deformations preserving the Zoll property on $\{H =-\frac{h}{2}\}$ and $\{H =-\frac{k}{2}\}$.
     \end{proof}
    \begin{figure}[h]
 	\centering
 	\begin{subfigure}{.5\textwidth}
 		\centering
 		\includegraphics[width=0.9\textwidth]{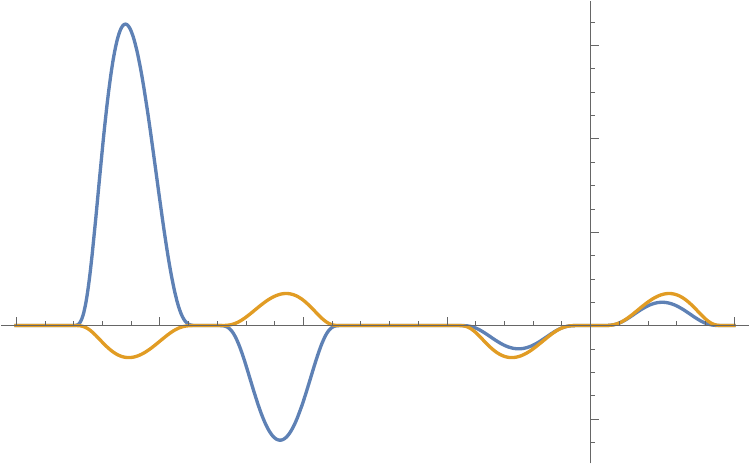}
 		%  \caption{}
 		% \label{fig:sub1}
 	\end{subfigure}%
 	\begin{subfigure}{.5\textwidth}
 		\centering
 		\includegraphics[width=0.9\textwidth]{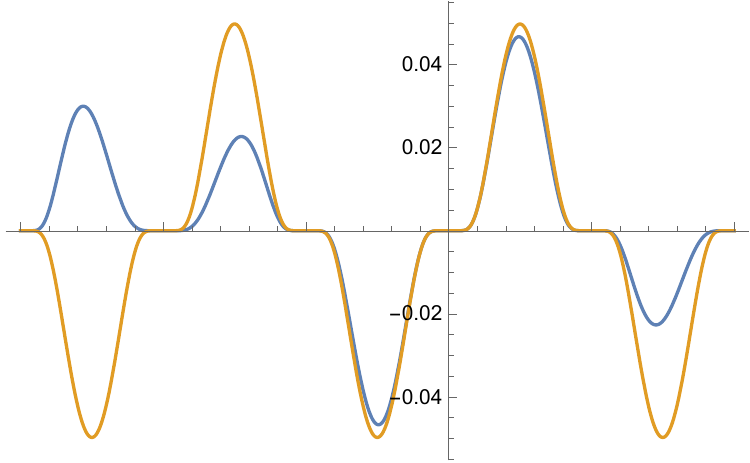}
 		% \caption{}
 		%\label{fig:sub2}
 	\end{subfigure}
 	\caption{A) The functions $\tilde f$ (in blue) and $G$ (in yellow) of {Case 1} of Theorem \ref{thm:deformation_fixed_energies} with $\xi = 3/2.$  B) The functions $\tilde f$ (in blue) and $G$ (in yellow) of {Case 3} of Theorem \ref{thm:deformation_fixed_energies} with $\xi = 1/2$.}
 	\label{fig:0}
 \end{figure}
     
    \begin{rmk} An equivalent of Corollary 2 cannot be proved in this setting, at least not with the same strategy of proof. In fact, projective transformations act by time reparametrization only on one energy level, while losing all dynamical information on the other ones. Therefore, applying a projective transformation to one of the systems $(g,V_{\text{Kep}})$ as in Theorem 2 one will get the Zoll property only at one energy losing it in general on the second one. 
    \end{rmk}

\begin{proof}[Proof of Theorem \ref{thm:existencearbitrary}]
   
   Let $0>H_{n-1}>H_{n-2}\dots>H_1>H_0$ be $n$ distinct energy levels and define $h_i = -2H_i$ accordingly. 
   By assumption we have $h_i\ge 2 h_{i-1}$ for all $i$. We thus fall into Case 1 of the proof of Theorem \ref{thm:deformation_fixed_energies} for each pair $(h_i,h_{i-1})$ and can conclude by induction on $n$.
   The various successive extensions are represented in a particular case in Figure \ref{fig:2}.
\end{proof}

    \begin{figure}[h]
\centering
\begin{subfigure}{.5\textwidth}
  \centering
	\includegraphics[width=0.9\textwidth]{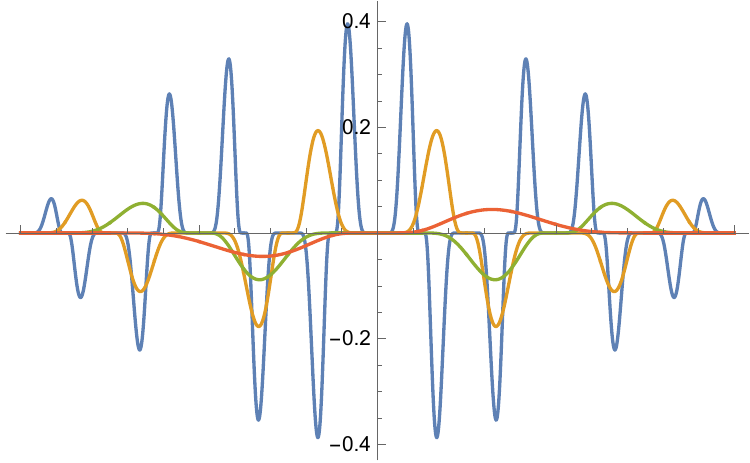}
%  \caption{}
 % \label{fig:sub1}
\end{subfigure}%
\begin{subfigure}{.5\textwidth}
  \centering
  	\includegraphics[width=0.9\textwidth]{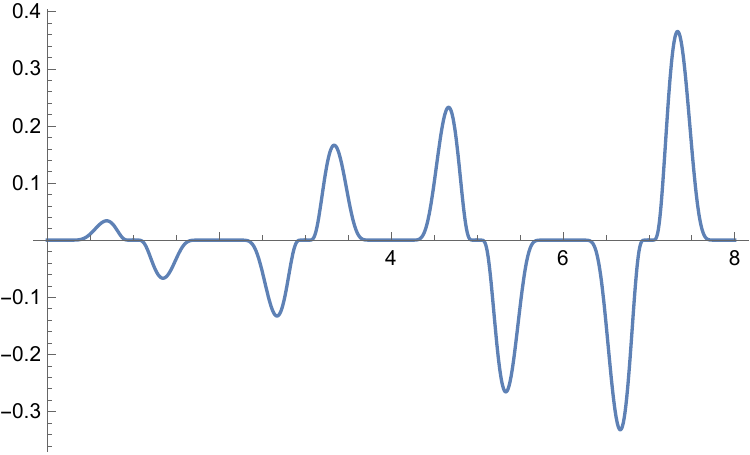}
 % \caption{}
  %\label{fig:sub2}
\end{subfigure}
\caption{A) Three successive extensions of a compactly supported odd function in $[-1,1]$, coloured in red, all drawn as functions on $[-1,1]$. The first extension $F_1$ is depicted in green, the second extension $F_2$ in orange and the third one, $F_3$, in blue. B) The function $F_3(1-k\rho)/(2-k\rho)$ plotted as a function of $\rho$ .}
\label{fig:2}
\end{figure}

\begin{rmk}
Figure \ref{fig:2} suggests that the $L^\infty$-norm of the extensions increases with the number of reflections. This intuition can be checked using the reflection laws \eqref{eq:parity_tilde_f} and \eqref{eq:ftildeodd} at each successive reflection. The details are left to the reader. In particular, the proof of Theorem \ref{thm:existencearbitrary} does not extend to an infinite number of energies. 
\end{rmk}

   %%%%%%%%%%%%%
   %%%%%%%%%%%%%
   %%%%%%%%%%%%%

\section{Rigidity and flexibility phenomena}
\label{sec:proof_rigidity}

\subsection{ Rigidity and flexibility for close energy values}
\label{sec:multipleclose}

As already mentioned in Section \ref{sec:more_levels}, the statement of Theorem \ref{thm:existencearbitrary} for the family given in \eqref{eq:deformationkepler} with $\varphi_\tau\equiv 0$ is in general false if the energy levels are too close to each other, and indeed in this case 
both rigidity and flexibility phenomena occur. We illustrate the situation in more detail in this section. 

Let us consider $h_1>h_2>\dots>h_n>0$ with $n\ge3$. 
Following the proof of Theorem  \ref{thm:deformation_fixed_energies}, we assume that a deformation as in \eqref{eq:deformed_lagrangian_metric} is given on a disk of radius $2/h_n$, and denote by $\tilde f$ any extension of $f$ to the interval $[1-\frac{2h_1}{h_n}, 1]$. Let $\mathcal{H}_i$ be the Hill's region corresponding to $h_i$. Using \eqref{eq:JM_metric_Kepler_h} and \eqref{eq:JM_metric_deformation_h} we see that $\tilde f$ induces the following Jacobi-Maupertuis metrics on $\mathcal H_i$:

    \begin{equation*}
 	g_{JM}^{h_i} = \frac{2-h_i\rho }{\rho}\left(\left(1+\frac{\tilde f \left(1-h_1\rho \right)(2-h_i\rho)}{(2-h_1\rho )(2-h_i\rho)}\right)^2d\rho^2+ \rho^2 d\theta^2\right).
    \end{equation*}
    
We set 
\begin{equation}
	\xi_i := \frac{h_1}{h_i}-1 , \quad \text{for} \quad  i=1\dots n,\label{eq:xi}
\end{equation}
and notice that $0=\xi_1<\xi_{2}<\dots <\xi_{n}<1$. As in the proof of Theorem \ref{thm:existencearbitrary}, a straightforward computation shows that the functions $F_i:[-1,1]\rightarrow \mathbb R$ defined by
\begin{align*}
F_{i} (1-h_i \rho) &:= \frac{\tilde f \left(1-h_1\rho \right)(2-h_i\rho)}{2-h_1\rho}, \quad \forall \rho \in [0,\frac{2}{h_i}],
\end{align*}
are odd if and only if $\tilde f$ satisfies the following reflection laws:
\begin{equation}
\tilde f (R_{-\xi_i}(x)) = - \tilde f(x) \frac{R_{-\xi_i}(x)^2 -1}{x^2-1}, \quad \forall x \in [1-2\xi_i,1], \label{refl:1}
\end{equation}
where $R_{-\xi_i}(x)$ stands for the reflections about $-\xi_i$.
Equations \eqref{refl:1} and 
\begin{equation}
\tilde f(-x) = - \tilde f(x), \quad \forall x \in [-1,1].
\label{refl:0}
\end{equation}
can be stated in a unified way by saying that the function $$G(x):=\frac{\tilde f(x)}{x^2-1}$$ should be odd with respect to the reflections $R_{-\xi_i}$ and $R_{\xi_1}$ about $-\xi_i$ and the origin respectively.
Thus, we are now imposing $n$ symmetry conditions on the function $\tilde f$ in order for the corresponding natural system to be Zoll at the energy levels $\{H=-\frac {h_i}{2}\}$. These conditions are necessary and sufficient. Notice indeed that, by \eqref{refl:1}, the functions $F_i$ are smooth and vanish at $-1$ and $1$ since $f$ vanishes at $1$. 

\subsubsection{Rationally dependent case}
Let us further assume that $\xi_2,...,\xi_n$ are (pairwise) rationally dependent, which is equivalent to saying that $\xi_i \in \mathbb{Q}\xi_2$ for all $i\ge 3.$ This implies that the group $\mathcal{G}$ generated by $\{R_{-\xi_1},\dots R_{-\xi_n}\}$ is discrete. Let us define
 \begin{equation}
 	\label{eq:def_gamma}
 \gamma:= \inf \left \{ a_2\xi_2 + ... + a_n\xi_n >0 \ |\ a_2,...,a_n \in \mathbb Z \right\} >0.
\end{equation}
Clearly, $\mathcal G$ contains the group generated by $\{R_{-\xi_1},R_{-\gamma}\}$ as a subgroup and thus we fall in the scope of Theorem \ref{thm:existencearbitrary}. In particular, the space of deformations preserving the Zoll property on all $\{H=-\frac{h_i}{2}\}$, $i=1,...,n$, contains a copy of $\mathscr{C}^\infty_{\mathrm c}(0,\gamma)$ and as such is infinite dimensional. Thus, we have shown the following

\begin{cor}
\label{cor:cor3}
	Let  $h_1>h_2>\dots>h_n>0$ be such that the corresponding $\xi_2,...,\xi_n$ defined by \eqref{eq:xi} are pairwise rationally dependent. Then, the space of non-equivalent rotationally invariant Riemannian metrics on the plane such that the natural system $(g,V_{\mathrm{Kep}})$ is Zoll on all $\{H=-\frac {h_i}{2}\}$ is infinite dimensional and contains a subset which is in bijection with the space $\mathscr C^\infty_{\mathrm c}(0,\gamma)$ of smooth functions having compact support in $(0,\gamma)$ where $\gamma$ is given in \eqref{eq:def_gamma}. 
\end{cor}

The next corollary is an immediate consequence of Corollary \ref{cor:cor3}.

\begin{cor}
	Let $h_1>h_2>\dots>h_n$ be fixed. For any $\ve>0$ there exist $h'_1>h'_2>\dots>h'_n$ with $\vert h_i'-h_i\vert<\ve$ for all $i=1,...,n$  and such that the space of non-equivalent rotationally invariant Riemannian metrics on the plane such that the natural system $(g,V_{\mathrm{Kep}})$ is Zoll on $\{H=-\frac {h'_i}{2}\}$ for all $i=1,...,n$ is infinite dimensional and contains a subset which is in bijection with the space $\mathscr C^\infty_{\mathrm c}(0,\gamma)$ of smooth functions having compact support in $(0,\gamma)$ where $\gamma$ is given in \eqref{eq:def_gamma}. 
\end{cor}

 An explicit example for the case $n=3$ and $(h_1,h_2,h_3) = (15,12,10)$ is depicted in Figures \ref{fig:3} and \ref{fig:4}.
 
      \begin{figure}[h]
 	\centering
 	\begin{subfigure}{.5\textwidth}
 		\centering
 		\includegraphics[width=0.9\textwidth]{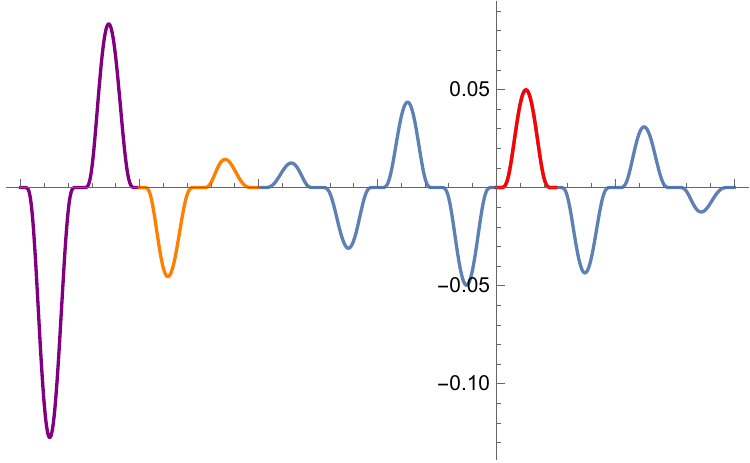}
 		%  \caption{}
 		% \label{fig:sub1}
 	\end{subfigure}%
 	\begin{subfigure}{.5\textwidth}
 		\centering
 		\includegraphics[width=0.9\textwidth]{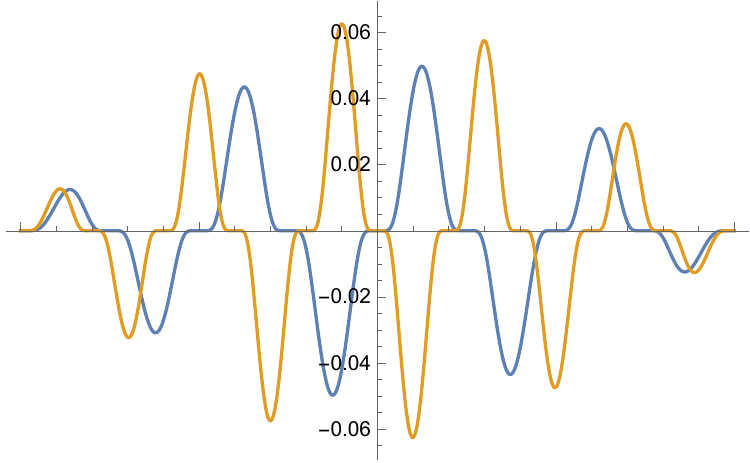}
 		%  \caption{}
 		%\label{fig:sub2}
 	\end{subfigure}
 	\caption{The figure on the left shows the extension $\tilde f$ on  $[-2,1]$. The original $f$ is coloured in blue. In this case, $\gamma=\frac 14$ and the smooth function compactly supported in $(0,1/4)$ is coloured in red. The extensions on $[-3/2,-1]$ and $[-2,-3/2]$ are drawn in orange and purple respectively. The figure on the right shows $F_2$ (in yellow) and $f$ (in blue).}
 	\label{fig:3}
 \end{figure}
 
 \begin{figure}[h]
 	\centering
 	\begin{subfigure}{.5\textwidth}
 		\centering
 		\includegraphics[width=0.9\textwidth]{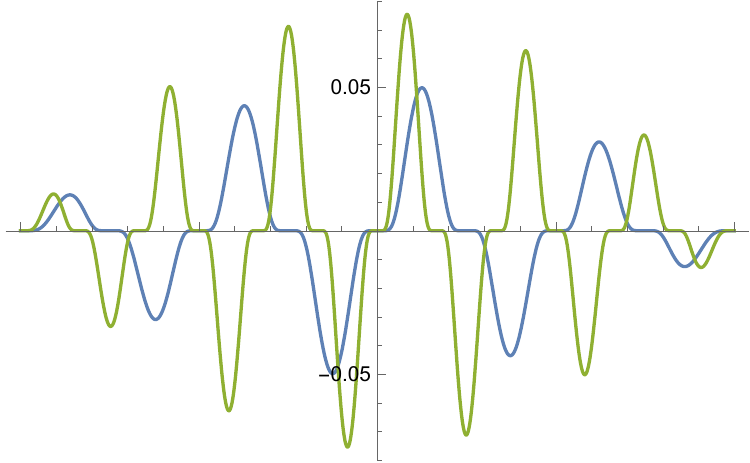}
 		%  \caption{}
 		% \label{fig:sub1}
 	\end{subfigure}%
 	\begin{subfigure}{.5\textwidth}
 		\centering
 		\includegraphics[width=0.9\textwidth]{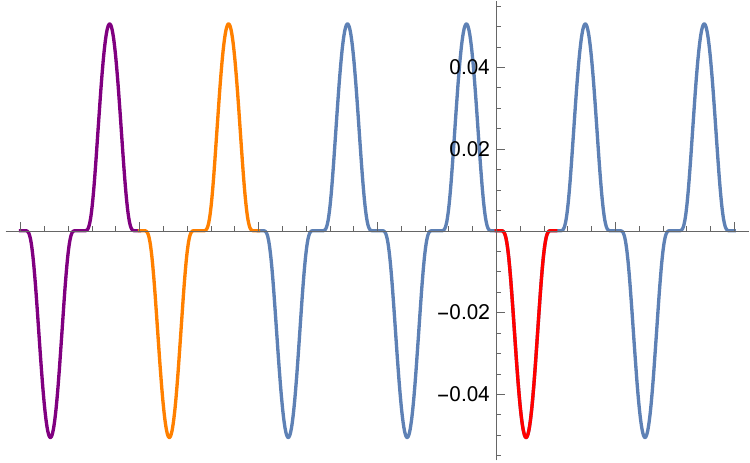}
 		% \caption{}
 		%\label{fig:sub2}
 	\end{subfigure}
 	\caption{The figure on the left shows $F_3$ (green) and $f$ (blue). The function $G(x)=\tilde{f}(x)/(x^2-1)$ is depicted on the right.}
 	\label{fig:4}
 \end{figure}

\subsubsection{Rationally independent case}
Here we consider the case in which at least two of the $\xi_i$, $i=2,...,n$, are rationally independent. Since the argument we are going to present does not involve the other energy levels we can assume without loss of generality that $n=3$ and set $h_{i_1} = \kappa> h_{i_2}= \ell$ and $h_1 = h$. Thus, assume  $\xi_\kappa$ and $\xi_\ell$ are rationally independent and additionally satisfy the following ``smallness'' condition
\begin{equation}
0< \xi_\kappa+\xi_\ell <1. \label{eq:smallness}
\end{equation}
We will show that \eqref{refl:1} and \eqref{refl:0} force $f$ to vanish identically on $[-1,1]$, thus also $\tilde f$ to vanish identically on $[1-\frac{2h}\ell,1]$. 

\begin{lemma}
Assume that $0<\ell<\kappa<h<2\ell$ are so chosen that $\xi_\ell$ and $\xi_\kappa$ as defined in \eqref{eq:xi} are rationally independent and satisfy condition \eqref{eq:smallness}. Then, the following holds: if
$\tilde f: [1-\frac{2h}\ell,1]\rightarrow \mathbb R$ is a continuous function satisfying \eqref{refl:1} and \eqref{refl:0}, then $\tilde f$ is identically zero. 
\label{lem:1}
\end{lemma}

\begin{proof}
By \eqref{refl:0}, $f$ (and hence $\tilde f$) must vanish at $x=0$. The idea now is that, using the reflections $R_0$, $R_{-\xi_\kappa}$ and $R_{-\xi_\ell}$, we can map $x=0$ onto a dense subset $\mathcal D \subset [-1,1]$ in such a way that, throughout the orbit, the point $x=0$ never leaves the interval $[-1,1]$. Equations \eqref{refl:1} and \eqref{refl:0}, which must hold in particular for all $x\in [-1,1]$, then imply that $f$ vanishes on $\mathcal D$ and hence on $[-1,1]$. 

To this purpose, notice that for all $\xi\in\mathbb R$ we have
$$R_0\circ R_{\xi} = T_{-2\xi}, \quad R_\xi\circ R_0 = T_{2\xi},$$
where $T_\nu:\mathbb R\rightarrow \mathbb R$ denotes the translation by $\nu\in\mathbb R$. This implies that the group $\mathcal G$ of isometries of $\mathbb R$ generated by $R_0$, $R_{-\xi_\kappa}$ and $R_{-\xi_\ell}$ contains the translations $T_{ 2 n\xi_\kappa}, T_{ 2 n \xi_\ell}$ for $n\in \mathbb{Z}$ as well as the reflections $R_{m \xi_\kappa}$ and $R_{m \xi_\ell}$ for $m\in \mathbb{Z}$.

Condition \eqref{eq:smallness} implies that $T_{\pm 2\xi_\kappa}(0)$ are again contained in the interval $[-1,1]$. Set for instance $x_1:= T_{2\xi_\kappa}(0)$.
Thanks to \eqref{eq:smallness},  at least one between 
$T_{2\xi_\kappa}(x_1)$ and $T_{-2\xi_\ell}(x_1)$ is again contained in $[-1,1]$.
Suppose that the second possibility occur (the argument in the first case being completely analogous), and observe that 
$$x_2 := T_{-2\xi_\ell}(x_1) = T_{2\xi_\kappa-2\xi_\ell}(0)\in [-1,0).$$
Now let $n_1\in \mathbb N$ be the only natural number such that
$$x_3:= T_{2n_1\xi_\kappa}(x_2) = T_{2(n_1+1)\xi_\kappa - 2 \xi_\ell}(0
)\in (0,x_1)$$ 
and set $\gamma_1:= (n_1+1)\xi_\kappa - \xi_\ell$. By construction, $0<\gamma_1<\xi_\kappa$, $T_{2\gamma_1}$ is in the group $\mathcal G$ of isometries generated by $R_0$, $R_{-\xi_\kappa}$ and $R_{-\xi_\ell}$, and the ``orbit'' $0,x_1,x_2,x_3$ is entirely contained in $[-1,1]$. Repeating the same argument starting from $x_3$ and replacing the translations $T_{\pm 2 \xi_\kappa}, T_{\pm 2\xi_\ell}$ with $T_{\pm 2 \gamma_1}, T_{\pm 2 \xi_\kappa}$ we obtain $0<\gamma_2<\gamma_1$ such that the orbit $0,x_1,x_2,...,x_6$ is entirely contained in $[-1,1]$. Iterating this procedure, and using the fact that $\xi_\kappa$ and $\xi_\ell$ are rationally independent, we obtain a sequence $\gamma_m\downarrow 0$ such that all $T_{2 \gamma_m}$ is in $\mathcal G$ for all $m\in \mathbb N$ and an ``orbit'' starting from $x=0$ which is dense in $[-1,1]$. The reflection laws \eqref{refl:1} and \eqref{refl:0} now imply that $f$ vanishes on the whole orbit, thus completing the proof.
\end{proof}

The next result is an immediate corollary of the lemma above.

\begin{cor}
Assume that $h_1>h_2>\dots>h_n$ are such that $\xi_i$ and $\xi_j$ as defined in \eqref{eq:xi} are rationally independent for some $i\neq j$ and satisfy the ``smallness'' condition \eqref{eq:smallness}. Then the following rigidity statement holds: if $g$ is a rotationally invariant Riemannian metric on the plane such that the natural system $(g,V_{\mathrm{Kep}})$ is Zoll on $\{H=-\frac {h_i}2\}$ for all $i$, then $g$ is equivalent to the flat metric. 
\end{cor}

\begin{rmk} 
As shown by Theorem \ref{thm:existencearbitrary}, some ``smallness'' condition is needed in order to achieve rigidity as in the corollary above. It is however not clear to us whether \eqref{eq:smallness} is optimal in this sense or not. 
\end{rmk}

 \bibliography{_biblio}
\bibliographystyle{plain}

\end{document}